\documentclass[11pt,a4paper]{amsart}
\usepackage[a4paper,centering,left=2.2cm,right=2.2cm]{geometry}
\usepackage{amsfonts,amscd,amssymb,amsmath,amsthm,hyperref,mathrsfs,multirow,xcolor,lscape,ifthen,mathtools,graphicx,fancyhdr,enumerate}
\usepackage[thinlines]{easytable}
\usepackage[all,cmtip]{xy}
\usepackage{tikz}
\usetikzlibrary{matrix}
\usetikzlibrary{calc}
\usepackage{url}
\usepackage{hyperref}

\newtheorem{thm}{Theorem}[section]
\newtheorem{lem}[thm]{Lemma}
\newtheorem{defn}[thm]{Definition}
\newtheorem{cor}[thm]{Corollary}

\newtheorem{example}[thm]{Example}

\numberwithin{equation}{section}
\newtheorem{remark}[thm]{Remark}
\newtheorem{note}[thm]{Note}

\newcommand{\mr}[1]{\mathrm{#1}}
\newcommand{\C}{\mathbb{C}}

\newcommand{\R}{\mathbb{R}}
\newcommand{\z}{\mathfrak{z}}
\newcommand{\h}{\mathbb{H}}
\newcommand{\X}{\mathbb{X}}

\newcommand{\Z}{\mathbb{Z}}
\newcommand{\Q}{\mathbb{Q}}
\newcommand{\n}{\mathfrak{n}}
\newcommand{\G}{\Gamma}
\newcommand{\m}{\mathrm{M}}
\newcommand{\E}{\mathcal{E}}
\newcommand{\Cu}{\mathfrak{C}}
\newcommand{\rG}{\mathrm{G}}
\newcommand{\e}{\mathfrak{e}}

\newcommand{\fL}{\mathfrak{L}}

\newcommand{\fg}{\mathfrak{g}}
\newcommand{\fl}{\mathfrak{l}}

\newcommand{\tq}{\widetilde{q}}
\newcommand{\Ii}{\mathrm{i}}
\def\lim{\mathrm{lim}}
\newcommand{\g}{\gamma}
\newcommand{\om}{\omega}
\def\ker{\mathrm{ker}}

\def\cu{{\mathfrak{c}}}

\def\exp{\mathrm{exp}}
\newcommand{\ie}{i.e.\ }
\newcommand{\tmt}[4]{\left({#1\atop #3}{#2\atop #4}\right)}
\newcommand{\cmt}[2]{\left({#1\atop #2}\right)}

\newcommand{\eqr}[1]{\mbox{(\ref{eq:#1})}}
\newcommand{\beq}{\begin{equation}}
\newcommand{\eeq}{\end{equation}}
\newcommand{\mc}[1]{\mathcal{#1}}
\newcommand{\mf}[1]{\mathfrak{#1}}

\begin{document}
\date{\today}
\title[Lifting of Modular Forms]{Lifting of Modular Forms}
\author{Jitendra Bajpai}
\address{Mathematisches Institut, Georg-August Universit\"at G\"ottingen, D-37073 Germany.}
\email{jitendra@math.uni-goettingen.de}
\maketitle

\begin{abstract}---
The existence and construction of vector-valued modular forms (vvmf) for any arbitrary  Fuchsian group $\rG$, for any representation $\rho:\rG \longrightarrow \mr{GL}_{d}(\C)$ of finite image can be established by lifting scalar-valued modular forms of the finite index subgroup $\ker(\rho)$ of $\rG$. In this article vvmf are explicitly constructed for any admissible multiplier (representation) $\rho$, see section~\ref{af} for the definition of admissible multiplier. In other words, the following question has been partially answered: \emph {For which representations $\rho$ of a given $\rG$, is there a vvmf with at least one nonzero component ?}\\
\vspace{0.5cm}

\noindent{R{\tiny{\'ESUM\'E}}\,.\,\,}--- L'existence et construction de formes modulaires vectorielles (vvmf) pour un groupe Fuchsien arbitraire $G$ et pour une repr\'esentation  $\rho : \rG \longrightarrow \mr{GL}_{d}(\C)$ d'image finie peut \^etre \'etablie en relevant des formes modulaires scalaires pour le sous-groupe d'indice fini $\ker (\rho)$ de $G$. Dans cet article, des vvmf sont  explicitement construites pour tout multiplicateur admissible (repr\'esentation) $\rho$ (voir paragraphe~\ref{af} pour la d\'efinition du multiplicateur admissible). En d'autres termes, on a partiellement r\'epondu \`a la question suivante: \emph {Pour quelles repr\'esentations $\rho$ d'un groupe $G$ donn\'e, existe-t-il une vvmf avec au moins une composante non nulle?}


 \end{abstract}


\section{Introduction}
Scalar-valued modular forms and their generalizations are one of the central concepts in number theory. Why is the notion of vector-valued modular forms one of the natural generalizations of scalar-valued modular forms? History of modern mathematics answers this question naturally. All of the most famous modular forms have a multiplier, for example: $$\eta\left(\frac{a\tau+b}{c\tau+d}\right)=\sqrt{c\tau+d} \cdot \rho \left({a \atop c}{b \atop d} \right)\cdot \eta(\tau) \qquad \mr{for} \quad  \left( {a\atop c} {b \atop d}\right) \in \mathrm{SL}_2(\Z), \quad \tau \in \h .$$ Here $\h=\{ z=x+\Ii y \in \C \ | \ y >0 \}$, denotes the upper half plane and $\eta(\tau)=q^{1/24}\Pi_{n=1}^{\infty}(1-q^{n})$ is the Dedekind eta function with $q=e^{2\pi \Ii \tau}$. In this case the multiplier $\rho$ is a $1$-dimensional representation of the double cover of SL$_2(\Z)$. These examples suggest having multipliers $\rho$ of dimension $d\geq 1$ and the corresponding modular forms are called {\bf{\emph{vector-valued modular forms (vvmf)}}}. 

\par In the 1960's, Selberg~\cite{Selberg} called for a theory of vvmf, as a way to study the noncongruence scalar-valued modular forms as these look intractable by the methods available in the theory of scalar-valued modular forms. In the 1980's, Eichler-Zagier~\cite{EZ} explained how Jacobi forms and Siegel modular forms for $\mr{Sp}(4)$ can be reduced to vvmf. Since then the theory has been in demand to be developed. The work of Borcherds and the rise of the string theory in physics have been major catalyst in the development of the theory. This theory of vvmf has applications in various fields of mathematics and physics such as vertex operator algebra, conformal field theory, Borcherds-Kac-Moody algebras, etc. In Zwegers' work~\cite{Zwegers02} on Ramanujan's mock theta functions, vvmf have played an important role to make them well fit in the world of modular forms. There are plenty of vvmf in ``nature". For instance the characters of a rational conformal field theory (RCFT) form a vvmf of weight zero, see~\cite{CIZ87,MMS88}. The Borcherds lift associates vvmf for a Weil representation to automorphic forms on orthogonal groups with infinite product expansions, which can arise as denominator identities in Borcherds-Kac-Moody algebras, see~\cite{Bruinier2002, Bruinier2014}.

\par In terms of developing the theory of vvmf, some efforts have been made to lift to vvmf, classical results  like dimension formulas and growth estimates of Fourier coefficients of vvmf of the modular group. For example we refer \cite{BG, KM2, Knopp4, Marks1, Marks3, M4} to mention a few of these efforts and~\cite{CF,Gannon1} for the current state of the art. However, this article is mainly concerned with Fuchsian groups of the first kind and looks further than the modular group regarding explicit construction of vvmf. The existence of vvmf for any Fuchsian groups of the first kind with respect to any multiplier has been discussed in~\cite{Sebbar}. The classification of vvmf for any genus zero Fuchsian groups of the first kind and a method to construct vvmf for triangle groups have been established by the author in his doctoral dissertation~\cite{BajpaiThesis}. 

\par More precisely, this article will show that a vvmf $\X(\tau)$ of a finite index subgroup $\mr{H}$ of any Fuchsian group of the first kind $\rG$ can be lifted to one of the vvmf $\widetilde{\X}(\tau)$ of $\rG$ by inducing the multiplier. Similarly a vvmf $\X(\tau)$ of $\rG$ can be restricted to one of the vvmf $\overline{\X(\tau)}$ of any of the finite index subgroup $\mr{H}$ by reducing the multiplier. However, lifting of a vvmf increases the rank of vvmf by the factor equal to the index of $\mr{H}$ in $\rG$ whereas the restriction does not affect the rank of vvmf.

\par These arguments give an easy construction of vvmf of any finite index subgroup $\mr{H}$ of $\rG$. The lifting argument can also be used to verify the existence of scalar-valued noncongruence modular forms. Usually, for any multiplier $\rho: \G(1) \rightarrow \mr{GL}_d(\C)$, $\ker \rho$ will be a noncongruence subgroup of  $\G(1)$.  Since all the components of vvmf of $\rG$ are scalar-valued modular forms of $\ker\rho$, this gives a different approach and direction to develop the theory of scalar-valued noncongruence modular forms of $\G(1)$ and hence some hope to contribute substantially in the development of the long standing Atkin-Swinnerton-Dyer conjecture about the unbounded denominator (ubd) property of modular forms of $\G(1)$. For original account of this problem, see~\cite{ASD}. To this date there are many advances have been witnessed to resolve and address the ubd property of noncongruence modular forms. For example~\cite{KL2, KL1, LiLong, Scholl2} and more recently by Franc and Mason in~\cite{FM1,FM2}. 

\par One of the advantages of vvmf is that (unlike scalar-valued modular forms) it is closed under inducing. For example $\theta_2(\tau)$ and $\eta(\tau)$ are scalar-valued modular forms of weight $1/2$ of $\G(2)$. However, $\theta_2(\tau)$ is not a scalar-valued modular form of $\G(1)$, but their lifts $\widetilde{\theta}_2(\tau),\ \widetilde{\eta}(\tau)$ are vvmf of $\G(1)$ with respect to the rank six multiplier $\widetilde{1}=\mr{Ind}_{_{\G(2)}}^{^{\G(1)}}(1)$. It is customary to denote the space of weight $w$ scalar-valued weakly holomorphic and holomorphic modular forms of group $\Gamma$ by $\m_{w}^{!}(\G):=\m^{!}_{w}(\G,1)$ and $\m_{w}(\G):=\m_{w}(\G,1)$ respectively. Throughout the article for any matrix $A$ of order $m\times n$, $A^{\mf{t}}$ denotes the transpose of $A$.

\par Throughout this article, we work with even integer weights - the same construction works for fractional weights but extra technicalities obscure the underlying ideas. It is also shown below that the spaces $\m^{!}_{w}(\G(1), \widetilde{1})$ and $\m^{!}_{w}(\G(1),1)$ are naturally isomorphic modules over the ring $\m^{{!}}_{0}(\G(1),1)$. More importantly, we have achieved $\m^{{!}}_{0}(\rG,1)$-module isomorphism between $\m^{!}_{w}(\rho)$ and $\m^{!}_{w}(\widetilde{\rho})$, see Theorem~\ref{induction}. In addition, for any admissible multiplier $\rho : \mr{H}\rightarrow \mr{GL}_d(\C)$ we prove the admissibility of $\widetilde{\rho}=\mr{Ind}_{_{\mr{H}}}^{^{\rG}}(\rho)$ in Theorem~\ref{adm-ind}. This construction was helpful in showing the existence of vvmf of any Fuchsian group $\rG$ of the first kind  for any finite image admissible multiplier $\rho$ and it is established in Theorem~\ref{thm:existence-for-fuchsian}. Among these results, Lemma~\ref{corep} explains a beautiful relation between the cusps of H and its index in G.

\par In the following section we review the theory of Fuchsian groups. There is vast literature available on Fuchsian groups and therefore nothing original is guaranteed in this section. For detailed exposition, see~\cite{Beardon, SK1,Shimura1, Venkov}. 

\section{Fuchsian Groups}
The study of Fuchsian groups begins by looking at the discrete group of motions of the upper half plane $\h$ in the complex plane $\C$ equipped with the Poincar\'e metric $ds^2=\frac{dx^2+dy^2}{y^2}$. The group of all orientation-preserving isometries of $\h$ for this metric coincides with the group PSL$_2(\R)=\mr{SL}_2(\R)/\{\pm I \},$ where $\mr{SL}_2(\R)=\big\{\tmt{a}{b}{c}{d} \ \big| \ a,b,c,d \in \R, ad-bc=1\big\}.$ Roughly speaking, a Fuchsian group is a discrete subgroup $\rG$ of PSL$_2(\R)$ for which $\rG\backslash \h$ is topologically a Riemann surface with finitely many punctures. 
The action of any subgroup of SL$_2(\R)$  on $\h$ is the M\"obius action $\left({a \atop c}{b \atop d}\right) \cdot \tau = \frac{a\tau+b}{c\tau+d}.$ Define $\h^{*}=\h \cup \R \cup \{\infty \}$ to be the extended upper half plane of $\mr{PSL}_2(\R)$ and this action can easily be extended to $\h^{*}$. The action of any $\g=\pm\tmt{a}{b}{c}{d} \in \mr{PSL}_2(\R)$,  on any $x \in \R \cup \{\infty\}$ is defined by $\g \cdot x =\mr{lim}_{\tau \mapsto x} \frac{a\tau+b}{c\tau+d} \in \R \cup \{ \infty\}.$
 
The elements of $\mr{PSL}_2(\R)$ can be divided into three classes: elliptic, parabolic and hyperbolic elements. An element $\g \in \mr{PSL}_2(\R)$ is elliptic, parabolic or hyperbolic, if the absolute value of the trace of $\g$ is respectively less than, equal to or greater than two.
Note that $\mr{PSL}_2(\R)$ fixes $\R \cup \{ \infty \}$ and, in $\h^{*}$ there is only one notion of $\infty$ usually denoted by $\Ii \infty$ but for notational convenience it will be written $\infty$. For any $x \in \R$ it is observed that there exists an element $\g=\pm\tmt{x}{-1}{1}{0}$ such that $\g \cdot \infty=x$ which means PSL$_2(\R)$ acts transitively on $\R \cup \{ \infty \}$. For any $x\in \R$ such $\g$ is denoted by $A_{_{x}}$.

\begin{defn}\rm
Let $\rG$ be a subgroup of $\mr{PSL}_2(\R)$. A point $\tau \in \h$ is called an elliptic fixed point of $\rG$ if it is fixed by some nontrivial elliptic element of $\rG$, and $\cu \in \R \cup \{\infty\}$ is called a cusp (respectively hyperbolic fixed point) of $\rG$ if it is fixed by some nontrivial parabolic (respectively hyperbolic) element  of $\rG$. Moreover, $\E_{_\rG}$ and $\Cu_{_\rG}$ denote the set of all elliptic fixed points and cusps of $\rG$ and define $\h_{_\rG}^{*}=\h \cup \Cu_{_\rG}$ to be the extended upper half plane of $\rG$.
\end{defn}

For example, if $\rG=\mr{PSL}_2(\R)$ then $\Cu_{_\rG}=\R \cup \{ \infty \}$, $\E_{_\rG}=\h$ and
if $\rG=\mr{PSL}_2(\Z)$ then $\Cu_{_\rG}=\Q \cup \{ \infty \}$,  $\E_{_\rG}=\rG\cdot \Ii \cup \rG\cdot \om$, where $\om=\frac{1+\Ii \sqrt{3}}{2}$. For any $\tau \in \h^{*}_{_\rG}$, let $\rG_{\tau}=\{ \g \in \rG | \g \cdot \tau =\tau \}$ be the stabilizer subgroup of $\tau$ in G. For each $\tau=x+\Ii y \in \h$, $\rG_{\tau}$ is a cyclic subgroup of G of finite order generated by $\g_{\tau}  = A_{\tau} K_{m} A_{\tau}^{-1}$, where $m=m(\tau)$ is the unique positive integer called the order of $\tau$, $A_{\tau}=\frac{1}{\sqrt{y}}\tmt{y}{x}{0}{1}$ such that $A_{\tau}(\Ii)=\tau$ and $K_{m}=\pm \tmt{\ \ \mr{cos}(\frac{\pi}{m})}{\ \mr{sin}(\frac{\pi}{m})}{-\mr{sin}(\frac{\pi}{m})}{\ \mr{cos}(\frac{\pi}{m})}$. For any $\cu \in \Cu_{_\rG} $, $\rG_{\cu}$ is an infinite order cyclic subgroup of G.  If $\cu =\infty$ then $\rG_{\infty}$ is generated by $\g_{\infty}  = \pm\left({1 \atop 0}{h_{_{\infty}} \atop 1}\right)=t^{h_{_{\infty}}}$ for a unique nonzero positive real number $h_{_{\infty}}$, called the cusp width of the cusp $\infty$, where we write $t=\pm \tmt{1}{1}{0}{1}$.  In case of $\cu \neq \infty$, $\rG_{\cu}$ is generated by $\g_{\cu}  = A_{\cu} t^{h_\cu} A_{\cu}^{-1}$ for some smallest nonzero positive real number $h_{\cu}$\,, called the cusp width of the cusp $\cu$ such that $\g_{\cu} \in \rG$ where $A_{\cu}= \pm\left({\cu \atop 1}{-1 \atop \ 0}\right) \in \mr{PSL}_2(\R)$ so that $A_{\cu}(\infty)=\cu$, as defined above. From now on for  convenience $h_{_{\infty}}$ will be denoted by $h$.

\subsection{Fuchsian groups of the first kind}  
The class of all Fuchsian groups is divided into two categories, namely Fuchsian groups of the first and of the second kind depending on the hyperbolic area of their fundamental domain. The fundamental domain, denoted by $\mr{F}_{_\rG}$, exists for any discrete group $\rG$ acting on $\h$. It is a connected open set $\mr{F}_{_\rG}$ in $\h$ in which no two elements of $\mr{F}_{_\rG}$ are equivalent with respect to $\rG$, and any point in $\h$ is equivalent to a point in the closure of $\mr{F}_{_\rG}$ with respect to $\rG$ \ie any $\rG$-orbit in $\h$ intersects with the closure of $\mr{F}_{_\rG}$. The hyperbolic area of $\mr{F}_{_\rG}$ may be finite or infinite. When $\mr{F}_{_\rG}$ has finite area then such $\rG$ is a Fuchsian group of the first kind otherwise of the second kind. For example $\rG=\big\langle \pm \tmt{1}{1}{0}{1}\big\rangle$ is the simplest example of a Fuchsian group of the second kind. 
\par  A Fuchsian group $\rG$ will have several different fundamental domains but this can be observed that their area will always be the same. From $\mr{F}_{_\rG}$ a (topological) surface $\Sigma_{_\rG}$ is obtained by identifying the closure $\widehat{\mr{F}}_{_\rG}$ of $\mr{F}_{_\rG}$ using the action of $\rG$ on $\widehat{\mr{F}}_{_\rG}$, \ie $\Sigma_{_\rG}= \widehat{\mr{F}}_{_\rG} / {\sim}$ (equivalently $\Sigma_{_\rG}= \rG \backslash \h^{*}_{_\rG}$). In fact $\Sigma_{_{\rG}}$ can be given a complex structure, for details see chapter 1 of~\cite{Shimura1}.   The surface $\Sigma_{_\rG}$ has genus-$\mathfrak{g}$ where as surface $\rG\backslash \h$ is of genus-$\mathfrak{g}$ with finitely many punctures. Due to Fricke, any Fuchsian group  $\rG$ of the first kind is finitely generated. In fact,  
$$\rG= \big\langle a_{i}, b_{i}, r_{j}, \g_{k}\ \big|\ \Pi_{i=1}^{\fg}[a_{i}, b_{i}] \cdot \Pi_{j=1}^{\fl} r_j \cdot \Pi_{k=1}^{\n} \g_{k} =1\,,\ r_{j}^{m_{j}}=1 \big\rangle $$

where $1\leq i \leq \fg, 1\leq j \leq \fl, 1\leq k \leq \n$ and $[a, b]=a b a^{-1} b^{-1}$. The elements $a_{i}, b_{i}$ are the generators of the stabilizer group of the $2\fg$ orbits of hyperbolic fixed points, each $r_{j}$ is the generator of the stabilizer group of $\fl$ orbits of elliptic fixed points, each $\g_{k}$ is the generator of the stabilizer group of $\n$ orbits of cusps of $\rG$ and for $j, \ m_{j}\in \Z_{\geq 2}\ $ denotes the order of elliptic element $r_j$.
\par The set of numbers $(\fg;m_1, \ldots, m_\fl; \n)$ is called the \emph{\textbf{signature}} of $\rG$. For example, the signature of $\G(1)$, $\G_{0}(2)$ and $\G(2)$ are $(0;2,3;1)$, $(0;2;2)$ and  $(0;\_ ;3)$ respectively, where `\_' represents the nonexistence of any nontrivial elliptic element in $\G(2)$. By using the Gauss-Bonnet formula, the area of any $\mr{F}_{_\rG}$ can be computed  in terms of its signature. Namely, $$\mr{Area}(\mr{F}_{_\rG})=2\pi \bigg[2\fg-2+\sum_{j=1}^{\fl}\bigg(1-\frac{1}{m_j}\bigg)+\n \bigg].$$ 
With respect to a set of generators of $\rG$, $\mr{F}_{_\rG}$ can be chosen to be the interior of a convex polygon bounded by $(4\fg+2\fl +2\n-2)$ geodesics, the sides of which are pairwise identified under the action of the generators of $\rG$. For $\rG=\G(1)$ the polygon is bounded by 4 sides, which can be seen in figure~\ref{fig:1}, although the sides $(\omega, \Ii)$ and $(\Ii, -\bar{\om})$ lie on the same geodesic. The group $\rG$ is called a co-compact group if $\n=0$. In addition if $\fl=0$ then the group $\rG$ is called a strictly hyperbolic group and in this case $\Sigma_{_\rG}$ is a compact Riemann surface of genus-$\fg$. 
In general, $\widehat{\mr{F}}_{_\rG}$ has exactly $\n$-vertices on $\R \cup \{\infty\}$. These vertices correspond to the inequivalent cusps of $\rG$.
\begin{figure}
\centering
\begin{tikzpicture}[scale=1.5]
\draw (0.5, 0.866) node[right]{$-\bar{\omega}$} arc [radius=1, start angle=60, end angle=120] node[left]{$\omega$}; \draw (-1.5,0)--(-0.5,0) node[below]{$-\frac{1}{2}$}; \draw (-0.5,0)--(0,0) node[below]{$0$}; \draw (0,0)--(0.5,0) node[below]{$\frac{1}{2}$}; \draw (0.5,0)--(1.5,0);
 \draw[dashed] (0, 1) node[above right]{$\Ii$} -- (0, 2); \draw (-0.5, 0.866)--(-0.5, 2); \draw (0.5, 0.866)--(0.5, 2);\end{tikzpicture}
\caption{Fundamental domain of $\G(1)$. All 4 geodesics can be described as follows: straight lines $\om$ to $\infty$ and $-\bar{\om}$ to $\infty$ contribute to two geodesics and the arcs $\om$ to $\Ii$ and $\Ii$ to $-\bar{\om}$ contribute to the other two geodesics, here $\om=\frac{-1+\Ii \sqrt{3}}{2}$.}\label{fig:1}
\end{figure}
\par One of the basic properties of Fuchsian groups of the first kind is that their action on $\h$ gives rise to a genus-$\mathfrak{g}$ surface $\Sigma_{_\rG}$ of finite area which give the Riemann surface of genus-$\mathfrak{g}$ with finitely many special points\,. These special points correspond to the $\rG$-orbits of elliptic fixed points and cusps of $\rG$. Let G be a such Fuchsian group of the first kind. Let $\widehat{\E}_{_\rG}:=\{\e_{j} \in \h \ | \ 1\leq j \leq \fl \}$ be a set of all inequivalent elliptic fixed points of G where for every $j, \e_{j}$'s are representatives of distinct orbits of elliptic fixed points of G with respect to its action on $\h$ and $\widehat{\Cu}_{_\rG}:=\{\cu_{k} \in \R \cup \{\infty\}\ |\ 1\leq k \leq \n \}$ be a set of all inequivalent cusps of G where for every $k, \cu_{k}$'s are representatives of the distinct orbits of cusps of G with respect to its action on $\R\cup \{\infty\}$\,. For example, if $\rG=\G(1),$ then $\widehat{\Cu}_{_\rG}=\{ \infty \}$ and $\widehat{\E}_{_\rG}=\{\Ii, \frac{1+\Ii\sqrt{3}}{2}\}$, if $\rG=\G_0(2)$, then $\widehat{\Cu}_{_\rG}=\{ 0,\infty \}$ and $\widehat{\E}_{_\rG}=\{\frac{1+\Ii}{2}\}$, if $\rG=\G(2)$, then $\widehat{\Cu}_{_\rG}=\{ 0,1,\infty \}$ and $\widehat{\E}_{_\rG}=\phi$. Consequently, $\rG\backslash \h_{_{\rG}}^{*}-(\widehat{\E}_{_\rG} \cup \widehat{\Cu}_{_\rG})$ is a Riemann surface with $\fl+\n$ punctures. 
\section{Vector-valued modular forms}\label{af}
Let $\rG$ denote a Fuchsian group of the first kind with a cusp at $\infty$, unless otherwise mentioned. More precisely, as long as $\rG$ has at least one cusp then that cusp can (and will) be moved to $\infty$ without changing anything, simply by conjugating the group by the matrix $A_{_{\cu}}=\pm\tmt{\cu}{-1}{0}{\ \ 1} \in \mr{PSL}_{2}(\R)$ if $\cu \in \R$ is a cusp of $\rG$. Roughly speaking a vvmf for $\rG$ of any weight $w\in 2\Z$ with respect to a multiplier $\rho$ is a meromorphic vector-valued function $\X: \h \rightarrow \C^{d}$ which satisfies a \emph{functional equation} of the form $\X(\g \tau)=\rho(\g)(c\tau+d)^{w}\X(\tau)$ for every $\g=\pm\tmt{a}{b}{c}{d} \in \rG$ and is also meromorphic at every cusp of $\rG$. The multiplier $\rho$ is a representation of $\rG$ of rank $d$ for arbitrary $d$ and is an important ingredient in the theory of vvmf. This article deals with the vvmf of $\rG$ of any even integer weight $w$ with respect to a generic kind of multiplier which we call an {\emph{\textbf{admissible multiplier}}}. This amounts to little loss of generality and is defined in the following  
\begin{defn}[Admissible Multiplier]\label{multiplier}\rm
Let $\rG$ be any Fuchsian group of the first kind with a cusp at $\infty$ and $\rho : \rG \rightarrow \mr{GL}_d(\C)$ be a rank $d$ representation of $\rG$. We say that $\rho$ is an admissible multiplier of $\rG$ if it satisfies the following properties:
\begin{enumerate}
\item $\rho(t_{\infty})$ is a diagonal matrix, i.e. there exists a diagonal matrix $\Lambda_{\infty} \in \mr{M}_d(\C)$ such that $\rho(t_{\infty})=\exp(2 \pi \Ii \Lambda_{\infty})$ and $\Lambda_{{\infty}}$ will be called an exponent matrix of cusp $\infty$. From now we fix an exponent matrix $\Lambda_\infty$ and it will be denoted by $\Lambda$.
\item $\rho(t_{\cu})$ is a diagonalizable matrix for every  $\cu \in \widehat{\Cu}_{_\rG}\backslash \{\infty\}$, \ie there exists an invertible matrix $\mathcal{P}_\cu \in \mr{GL}_d(\C)$ and a diagonal matrix $\Lambda_{\cu} \in \mr{M}_d(\C)$ such that $\mathcal{P}_{\cu}^{-1}  \rho(t_{\cu})\mathcal{P}_{\cu}= \exp(2 \pi \Ii \Lambda_{\cu})$, and $\Lambda_{{\cu}}$ will be called an exponent matrix of cusp $\cu$.
\end{enumerate}
\end{defn}
\begin{note}\rm
Note that each exponent $\Lambda_{_{\cu}}$ for every $\cu \in \Cu_{_{\rG}}$, is defined only up to changing any diagonal entry by an integer and therefore $\Lambda_{_{\cu}}$  is defined to be the unique exponent satisfying $0 \leq (\Lambda_{_{\cu}})_{\xi \xi} < 1$ for all $1\leq \xi \leq d$.  All modular forms have an infinite series expansion at the cusp $\cu \in \widehat{\Cu}_{_\rG}$. These expansions will be referred to as Fourier series expansions. Often these expansions are also referred to as $\tq_{_z}$- expansion with respect to $z \in \h_{_{\rG}}^{*}$ where in case of $z \in \E_{_{\rG}} \cup \Cu_{_{\rG}}$
$$ \tq_{_z}=\left\{\begin{array}{cccccccc}  
\exp\bigg(\frac{2\pi \Ii A_{_{z}}^{-1}\tau}{h_{z}}\bigg) & \mr{if} \ z \in \Cu_{_{\rG}}\\  
\big(\frac{\tau-z}{\tau-\bar{z}} \big)^{\ell} & \mr{if }\ z \in \E_{_{\rG}} \mr{ of \ order }\ \ell
\end{array}\,. \right .$$
\end{note}

\begin{remark}\label{admissible-remark}\rm$\;$
\begin{enumerate}
\item Dropping the diagonalizability does not introduce serious complications. The main difference is the Fourier coefficient in $\tq_{_{z}}$-expansions  become polynomials in $\tau$. A revealing example of such a vvmf is $\X(\tau)=\cmt{\tau}{1}$ of weight $w=-1$ for any $\rG$ with respect to the multiplier $\rho$ which is the defining representation of $\rG$. 
\item Obviously, if $\rho(t_{\infty})$ was also merely diagonalizable, $\rho$ could be replaced with an equivalent representation satisfying the assumption (1) of the multiplier system. Thus in this sense, assumption (1) is assumed without the loss of generality for future convenience.  Since almost every matrix is diagonalizable, the generic representations are admissible. For example: the rank two admissible irreducible representations of $\G(1)$ fall into 3 families parameterized by 1 complex parameter, and only six irreducible representations are not admissible. For details see section~4 of~\cite{Gannon1}.

\item The reason for assumptions (1) and (2) in the Definition~\ref{multiplier} of the multiplier system is that any vvmf  $\X(\tau)$ for $\rho$ will have $\tq_{_{\cu}}$-expansions.
\end{enumerate}
\end{remark}

\begin{defn}\label{mvvaf}\rm
Let $\rG$ be any Fuchsian group of the first kind with a cusp at $\infty$, $w \in 2\Z$ and $\rho: \rG\rightarrow \mr{GL}_{d}(\C)$ be any rank $d$ admissible multiplier of $\rG$. Then a meromorphic vector-valued function $\X : \h \rightarrow \C^{d}$ is a \emph{meromorphic vvmf} of weight $w$ of $\rG$ with respect to multiplier $\rho$, if $\X(\tau)$ has finitely many poles in $\widehat{\mr{F}}_{_{\rG}} \cap \h$ and has the following functional and cuspidal behaviour.
\begin{enumerate}
{\bf \item Functional behaviour} $$ \X(\g\tau) = \rho(\g) j(\g,\tau )^{w} \X(\tau), \quad \forall \g \in \rG \ \& \ \forall \tau \in \h,$$
{\bf \item Cuspidal behaviour}\\
\begin{enumerate}
{\item{\underline{at the cusp $\infty$}}:} $$ \X(\tau) = \tq^{\ \Lambda}\sum_{n=-M}^{\infty} \X_{[n]} \ \tq^{\ n}, \ \X_{[n]} \in \C^{d},$$ 
{\item{\underline{at the cusp $\cu (\neq \infty)$}}:}   $$  \X(\tau) = (\tau - \cu)^{-w} \mc{P}_{\cu}\tq^{\ \Lambda_{\cu}}\mc{P}_{\cu}^{^{-1}} \sum_{n=-M_{\cu}}^{\infty}  \X^{^{\cu}}_{[n]} \ \tq_{_{\cu}}^{\ n}, \ \X^{^{\cu}}_{[n]} \in \C^{d}.$$ 
\end{enumerate}
\end{enumerate}
\end{defn}

Following this weakly holomorphic and holomorphic vvmf of even integer weight with respect to an admissible multiplier is now defined.  
\begin{defn}\label{maindefn2}\rm
Let $\rG$ be any Fuchsian group of the first kind\,, $\rho$ be an admissible multiplier of $\rG$ of rank $d$ and $w \in 2\Z$. Then
\begin{enumerate}
\item A meromorphic vvmf $\X(\tau)$ is said to be \textbf{weakly holomorphic vvmf} for $\rG$ of weight $w$ and multiplier $\rho$ if $\X(\tau)$ is holomorphic throughout $\h$. Let $\m_{w}^{!}(\rho)$ denote the set of all such weakly holomorphic vvmf for $\rG$ of weight $w$ and multiplier $\rho$.

\item $\X(\tau) \in \m_{w}^{!}(\rho)$ is called a \textbf{holomorphic vvmf} if $\X(\tau)$ is holomorphic throughout $\h_{_{\rG}}^{*}$. Let $\m_{w}(\rho)$ denote the set of all such holomorphic vvmf for $\rG$ of weight $w$ and multiplier $\rho$.
 \end{enumerate}
\end{defn}

\begin{remark}\label{ring-RG}\rm
Let $R_{_\rG}$ denote the ring of scalar-valued weakly holomorphic automorphic functions of $\rG$. Then $R_{_{\rG}}:=\m_{0}^{!}(1)=\C[\mr{J}_{_{\rG}}^{^{\cu_1}}, \ldots, \mr{J}_{_{\rG}}^{^{\cu_\n}}]$, where $\n$ is the number of elements in the set $\widehat{\Cu}_{_{\rG}}$ and $\mr{J}_{_{\rG}}^{^\cu}$ is the normalized hauptmodul of $\rG$ with respect to $\cu \in \widehat{\Cu}_{_{\rG}}$. There is an obvious $R_{_\rG}$-module structure on $\m_{w}^{!}(\rho)$. Without loss of generality we may assume that $\cu_1=\infty$.
\end{remark}


\section{Lifting of modular forms}\label{lift} Let G be any Fuchsian group of the first kind with a cusp at $\infty$ and H be any finite index subgroup of G. In this section the relation between weakly holomorphic vvmf of H and G is established. Let $\n$ be the number of inequivalent cusps and $\fl$ be the number of inequivalent elliptic fixed points of $\rG$, and let $m_{j}, 1\leq j \leq \fl$, denote the orders of the elliptic fixed points. The weight of the cusp form $\Delta_{_\rG}(\tau)$ is $2\, \fL$ where $\fL=\mr{lcm}[m_{j}, 1\leq j \leq \fl]$. Think of it as the analogue for $\rG$ of the cusp form $\Delta(\tau) = q\prod_{n=1}^\infty(1-q^n)^{24}$ of $\G(1)$ of weight $12=2 \,\mr{lcm}[2,3].$ Like $\Delta(\tau)$, $\Delta_{_{\rG}}(\tau)$ is holomorphic throughout $\h_{_{\rG}}^{*}$ and nonzero everywhere except at the cusp $\infty$. Because $\Delta_{_{\rG}}(\tau)$ is holomorphic and nonzero throughout the simply connected domain $\h$, it possesses a holomorphic logarithm $\log \Delta_{_{\rG}}(\tau)$ in $\h$. For any $w \in \C$ define $\Delta_{_{\rG}}(\tau)^{w}=\exp(w\ \log\Delta_{_{\rG}}(\tau))$ then  $\Delta_{_{\rG}}(\tau)^{w}$ is also holomorphic throughout $\h$. A little work shows that it is a holomorphic modular form. For any $\rG$, let $\nu : \rG \rightarrow \C^{\times}$ denote the multiplier of the scalar-valued modular form $\Delta_{_\rG}(\tau)^{^{\frac{1}{2\fL}}}$. For example, in case of $\rG=\G(1)$ the multiplier $\nu$ for any $\g=\pm\tmt{a}{b}{c}{d}\in \G(1)$ is explicitly defined as follows:
\begin{equation*}
\nu(\g) = \left\{
\begin{array}{rl}
\exp[2\pi \Ii (\frac{a+d}{12c})-1-2 \sum_{i=1}^{c-1}\frac{i}{c}(\frac{di}{c}-\lfloor \frac{di}{c}\rfloor -\frac{1}{2})] & \text{if } c\neq 0\\
\exp[2\pi \Ii (\frac{a(b-3)+3}{12})] & \text{if } c = 0
\end{array} \right. 
\end{equation*}
More details on the multiplier system associated to $\eta(\tau)$ can be found in~\cite{Apostol,Knopp} and in general to any modular forms in~\cite{Rankin}. 
$\Delta_{_{\rG}}(\tau)$ is  obtained through exploiting its close connection with quasimodular form  $E_{(2,\rG)}(\tau)$ of weight 2 and depth 1 of group $\rG$. It is a solution to the equation $$\tq \frac{d }{d\tq} \Delta_{_{\rG}}(\tq)= \alpha \Delta_{_{\rG}}(\tq) E_{(2,\rG)}(\tq)$$ for some constant $\alpha$. For any $\rG$ with at least one cusp, the quasimodular forms $E_{(2,\rG)}(\tau)$ are discussed in~\cite{Zemel} and the cusp forms like $\Delta$ in~\cite{Cummins}. Using the above technical information, we obtain the following 
\begin{lem}\label{weight0 relation}
For any $w \in 2\Z$\,, $\m^{!}_{w}(\rho)$ and $\m^{!}_{0}(\rho \otimes \nu^{-w})$ are naturally isomorphic as $R_{_{\rG}}$-modules, where the isomorphism is defined by $\X(\tau) \mapsto \Delta_{_\rG}(\tau)^{^{-\frac{w}{2\fL}}}\X(\tau)$.
\end{lem}
We now state the two theorems which are the main results and focus of this section. 
\begin{thm}\label{adm-ind}
Let $\rG$ be any Fuchsian group of the first kind and $\mr{H}$ be any finite index subgroup of $\rG$, \ie $[\rG:\mr{H}] = m$. If $\rho$ is a rank $d$ admissible representation of $\mr{H}$ then the induced representation $\widetilde\rho = \mr{Ind}^{^{\rG}}_{_\mr{H}} (\rho)$ of $\rG$ of rank $dm$ is also an admissible representation.  \end{thm} 

\begin{thm}\label{induction}
Let $\rG$, $\mr{H}$, $\rho$ be as in theorem~\ref{adm-ind} and $w$ be an even integer. Then there is a natural $\mr{R}_{_\rG}$-module isomorphism between $\m^{!}_{w}(\rho)$ and $\m^{!}_{w}(\widetilde{\rho})$ where the induced representation $\widetilde{\rho} = \mr{Ind}^{^\rG}_{_\mr{H}} (\rho)$ is an admissible representation of $\rG$ of rank $dm$.
\end{thm}
Theorem~\ref{adm-ind} is an important tool to prove Theorem~\ref{induction} which establishes the relation between weakly holomorphic vvmf of H and G. More importantly, the isomorphism between $\m^{!}_{w}(\rho)$ and $\m^{!}_{w}(\widetilde{\rho})$ is given by $$\X(\tau) \mapsto \bigg( \X(\g_{1}^{-1}\tau), \X(\g_{2}^{-1}\tau), \cdots,  \X(\g_{m}^{-1}\tau)\bigg)^{\mathfrak{t}},$$ where $\{ \g_1, \cdots, \g_m\}$ are distinct coset representatives of H in $\rG$.

Before giving the proofs of Theorems~\ref{adm-ind} and~\ref{induction} let us recall why $\widetilde\rho = \mr{Ind}^{^\rG}_{_\mr{H}} (\rho)$ defines a representation. Write $\rG= \g_{1}\mr{H} \cup \g_{2}\mr{H} \cup \cdots \cup \g_{m}\mr{H}$. Without loss of generality we may assume that $\g_1=1$. The representation $\rho :\mr{H} \rightarrow \mr{GL}_{d}(\C)$ can be extended to a function on all of $\rG$, \ie $\rho: \rG \longrightarrow \mr{M}_{d}(\C)$ by setting $\rho(x) =0,  \forall x \notin \mr{H}$ where  $\mr{M}_{d}(\C)$ is the set of all $d\times d $ matrices over $\C$. The induced representation $\widetilde\rho =\mr{Ind}^{^{\rG}}_{_\mr{H}} (\rho) : \rG \longrightarrow \mr{GL}_{dm}(\C)$ is defined by \begin{equation}\label{eq:indrep} \widetilde\rho(x) = \left( \begin{array}{cccc}
\rho(\g_{1}^{-1} x \g_{1}) & \rho(\g_{1}^{-1} x \g_{2}) & \ldots  &\rho(\g_{1}^{-1} x \g_{m})\\
\rho(\g_{2}^{-1} x \g_{1}) & \rho(\g_{2}^{-1} x \g_{2}) & \ldots & \rho(\g_{2}^{-1} x \g_{m}) \\
\vdots & \vdots & \ddots& \vdots \\
\rho(\g_{m}^{-1} x \g_{1}) & \rho(\g_{m}^{-1} x \g_{2}) & \ldots  &\rho(\g_{m}^{-1} x \g_{m})\\
\end{array} \right), \quad \forall x \in \rG. \end{equation} 
Due to the extension of $\rho$ for any $x \in \rG$ and $\forall 1\leq i \leq m$ there exists a unique $1\leq j \leq m$ such that $\rho(\g_{i}^{-1} x \g_{j}) \neq 0$. Therefore, exactly one nonzero $d \times d$ block appear in every row and every column of~\eqr{indrep}. 

Before going into the details of the proofs of the Theorems~\ref{adm-ind} and~\ref{induction}, let us confirm that $\widetilde\rho$ does not depend on the choice of the coset representatives. 
\begin{lem}\label{cosetrep}
Let $\rG, \mr{H}$ be as in Theorem~\ref{adm-ind}. Let $\widetilde{R}=\{ g_1, \cdots, g_m\}$, $\widehat{R}=\{\g_1, \cdots, \g_m \}$ be two different coset representatives of $\mr{H}$ in $\rG$. Let $\rho: \mr{H}\rightarrow \mr{GL}_d(\C)$ be an admissible representation. Then the induced representation $\widetilde{\rho}=\mr{Ind}_{_{\mr{H}}}^{^{\rG}}(\rho)$ and $\widehat{\rho}=\mr{Ind}_{_{\mr{H}}}^{^{\rG}}(\rho)$ with respect to the coset representatives $\widetilde{R}$ and $\widehat{R}$ respectively are equivalent representations of $\rG$. 
\end{lem}

\begin{proof}  For each $1\leq i \leq m$ there exists $x_i \in \mr{H}$ such that $\g_i =g_i x_i $ up to reordering $g_i$'s and $\g_i$'s\,. Then $\widehat{\rho}(g)=D^{-1}\widetilde{\rho}(g)D$ for every $g \in \rG$ where block diagonal matrix  $D=\mr{Diag}\big( \rho(x_{1}), \cdots , \rho(x_{m})\big)$ is the conjugating matrix between $\widetilde{\rho}$ and $\widehat{\rho}$\,.\end{proof}

\begin{lem}\label{corep}
Let $\rG$ and $\mr{H}$ be as in Theorem~\ref{adm-ind}. Fix any cusp $\cu \in \widehat\Cu_{_{\rG}}$ and let $\cu_{1}, \cdots, \cu_{\n_{_{\cu}}}$ be the representatives of the $\mr{H}$-inequivalent cusps which are $\rG$-equivalent to the cusp $\cu$, so \beq\label{eq:c-ineq} \rG\cdot \cu=\cup_{i=1}^{\n_{_{\cu}}} \mr{H}\cdot{\cu_{i}} \,.\eeq Let $k_{_{\cu}}$ be the cusp width of $\cu$ in $\rG$ and $h_{\cu_{i}}$ be the cusp width of $\cu_{i}$ in $\mr{H}$. Write $h_{i}=\frac{h_{\cu_{i}}}{k_{_{\cu}}} \in \Z$, $\rG_{_{\cu}}=\langle t_{\cu}\rangle$ and $A_{i}(\cu)=\cu_{i}$ where $A_{i}\in \rG$. Then $m=\sum_{i} h_{i}$ and coset representatives of $\mr{H}$ in $\rG$ can be taken to be $g_{ij}=t_{\cu}^{j} A_{i}^{^{-1}}$ for all $i$ and $0\leq j < h_{i}$.
\end{lem}

\proof Let $g$ be any element of $\rG$. Because of the decomposition~\eqr{c-ineq} there is a unique $i$ such that $g^{-1} \cu= \g \cdot \cu_{i}$ for some $\g \in \mr{H}$. Then $A_{i} g \g$ fixes $\cu_{i}$ and so it equals $A_{i} t_{\cu}^{j} A_{i}^{-1}$ for some $j \in \Z$. Recall that $t_{\cu}=A_{\cu} t^{k_{\cu}} A_{\cu}^{-1}$ where $A_{\cu} =\tmt{\cu}{-1}{1}{0}\in \mr{PSL}_2(\R)$ such that $A_{\cu}(\infty)=\cu$. Note that $A_{i} t_{\cu}^{j} A_{i}^{-1}$ and $A_{i} t_{\cu}^{j+h_{i}} A_{i}^{-1}= A_{i} t_{\cu}^{j} A_{i}^{-1}A_{i} t_{\cu}^{h_{i}} A_{i}^{-1}$ lie in the same coset of $\mr{H}$ because $h_{i}$ is the least positive integer such that $A_{i} t_{\cu}^{h_{i}} A_{i}^{-1}\in \mr{H}$. Moreover $\mr{H}_{\cu_{i}}=\langle t_{i}= A_{i} t_{\cu}^{h_{i}} A_{i}^{-1} \rangle$ and $t_{i}=(A_{i}A_{\cu})t^{k_{\cu} h_{i}} (A_{i} A_{\cu})^{-1}$.  Thus we can restrict $0\leq j < h_{i}$. This means that every coset $g\mr{H}$ of $\mr{H}$ in $\rG$ contains an element of the form $t_{\cu}^{j} A_{i}^{-1}:=g_{ij}$ for some $0 \leq j < h_{i}$ and some $1 \leq i \leq \n_{\cu}$. This implies that $m \leq \sum_{i=1}^{\n_{\cu}} h_{i} $. 

In addition, for all the ranges of $i,j$ as above the cosets $g_{ij}\mr{H}$ are distinct. Indeed, let $i,j,k,l$ in the range as above such that $g_{ij}\mr{H}=g_{kl}\mr{H}$, \ie $t_{\cu}^{j} A_{i}^{-1}\mr{H}=t_{\cu}^{l} A_{k}^{-1}\mr{H}$, \ie $A_{k} t_{\cu}^{j-l} A_{i}^{-1} \in \mr{H}$\,. Then $ A_{k} t_{\cu}^{j-l} A_{i}^{-1}\cdot \cu_{i}= \cu_{k}$. Hence $\cu_{i}$ and $\cu_{k}$ are $\mr{H}$-equivalent cusps which implies that $i=k$. This implies that $A_{i} t_{\cu}^{j-l} A_{i}^{-1} \in \mr{H}_{\cu_{i}}$. $h_{i}$ is the smallest positive integer for which $A_{i} t_{\cu}^{h_{i}} A_{i}^{-1} \in \mr{H}$ and $0 \leq j, l< h_i$ therefore $0\leq |j-l| < h_i$ and $A_{i} t_{\cu}^{j-l} A_{i}^{-1} \in \mr{H}$ is possible only when $j-l=0$. This implies that $j=l$. Hence for $i,j,k,l$ ranged as above $g_{ij}\mr{H}=g_{kl}\mr{H}$ requires $i=k$ and $j=l$. Thus $\sum_{i=1}^{\n_{\cu}} h_{i}\leq m$ and we are done. \\
\qed

\begin{example}\rm

As an illustration of Lemma~\ref{corep}, Table~\ref{cuspindex} shows data for certain finite index subgroups $\mr{H}$ of $\rG=\G(1)$. In this case $\widehat{\Cu_{_{\rG}}}=\{ \infty \}$, $k_{\cu}=1$ and $\cu_{i} \in \widehat{\Cu_{_\mr{H}}}$, $h_i=h_{\cu_{i}}$.
\begin{table}[!htbp]
\centering
{\footnotesize
\begin{tabular}{|c|c|c|c|c|c|c|c|}
\hline $\mr{H}$ & $m$ &$\widehat{\Cu_{_\mr{H}}}$ & $h_{i}$ \\[2ex]
\hline\hline
$\G_0(2)$ & $3$ &$\{0, \infty\}$ &$2,1$ \\[2ex]
\hline
$\G(2)$& 6&$\{ 0,1, \infty\}$&$2,2,2$\\[2ex]
\hline
$\G_0(3)$ & $4$ &$\{0, \infty\}$ &$3,1$\\[2ex]
\hline
$\G(3)$& 12&$\{ -1, 0,1, \infty\}$&$3,3,3,3$\\[2ex]
\hline
$\G_0(4)$ & $6$ &$\{-\frac{1}{2}, 0, \infty\}$ &$1,4,1$ \\[2ex]
\hline
$\G(4)$ & $24$ &$\{-1, -\frac{1}{2}, 0, 1, 2, \infty\}$ &$4,4,4,4,4,4$\\[2ex]
\hline
$\G_{0}(8)$ &12&$\{-\frac{1}{4}, -\frac{1}{2}, 0, \infty\}$ &$1,2,8,1$  \\[2ex]
\hline 
\end{tabular}
} 
\vspace{3mm}
\caption{Relation between index and cusp widths of $\mr{H}$ in $\rG$.}\label{cuspindex}
\end{table}
\end{example}

\begin{proof}[Proof of Theorem~\ref{adm-ind}]  
We need to show that for each cusp $\cu$ of $\rG$, $\widetilde{\rho}(t_{\cu})$ is diagonalizable. Let $\cu$ be any cusp of $\rG$. Due to Lemma~\ref{cosetrep}, it is sufficient to choose the coset representatives as in Lemma~\ref{corep}. Then $\widetilde{\rho}(t_\cu)$ can be written in block form as 
\begin{equation}\label{eq:gij} 
\rho(g_{i j}^{-1} t_{\cu} g_{k l}) = \left\{\begin{array}{ccccccc}
 I\,, &\text{if } \ i=k \ \mr{and}\ j\ne h_{i} -1\\
 \\
 \rho(t_i)\,, & \text{if } \ i=k, \ j=0 \ \mr{and}\ l=h_{i} -1\\
 \\
 0\,,  &\text{otherwise} 
\end{array} \,\, , \right .
\end{equation}

where $i,j,k,l$ range as in Lemma~\ref{corep}, $I$ is the identity matrix of order $d\times d$ and $t_i = A_i t_{\cu}^{h_i} A_{i}^{-1}$ is the generator of the stabilizer $\mr{H}_{\cu_{i}}$ in $\mr{H}$. Thus $\widetilde{\rho}(t_\cu)$ is in block form, one for each $i$ of order $dh_{i} \times dh_{i}$.  Also, for every $1\leq i \leq\n_{\cu}$ $\rho(t_i):=T_{i}$ is diagonalizable by the admissibility hypothesis. So, for every $i$ let $v_{(i,k)}, 1\leq k \leq d$, be a basis of eigenvectors respectively with eigenvalues $\lambda_{(i,k)}$ of $\rho(t_i)$. Let $\zeta$ be any $h_{i}^{^{th}}$ root of unity and let $V_{(i,k,\zeta)}$ be the column vector of order $dm\times 1$, defined as follows. Its nonzero entries appear only in the $i^{th}$ block of order $dh_{i}\times 1$. That block is given by $\big(\lambda_{_{(i,k)}}^{^{1/h_{i}}} v_{(i,k)}, \zeta \lambda_{_{(i,k)}}^{^{1/h_{i}}} v_{(i,k)}, \zeta^{2}\lambda_{_{(i,k)}}^{^{1/h_{i}}} v_{(i,k)},\cdots, \zeta^{h_i -1} \lambda_{_{(i,k)}}^{^{1/h_{i}}} v_{(i,k)}\big)^{\mathfrak{t}} $.  From~\eqr{gij} it is clear that $V_{(i,k,\zeta)}$ is an eigenvector of $\widetilde{\rho}(t_{\cu})$ with eigenvalue $\zeta \lambda_{_{(i,k)}}^{^{1/h_i}}$. Hence, for every $ i $ there are exactly $dh_{i}$ eigenvectors of order $dm \times 1$ formed with respect to the $dh_{i}$ eigenvalues $\zeta \lambda_{_{(i,k)}}^{^{1/h_i}}$ for $\zeta=\exp\big(\frac{2\pi \Ii j}{h_i}\big)$ with $0\leq j < h_i$. Since $V_{_{(i,k ,\zeta)}}$ are linearly independent, $\widetilde{\rho}(t_\cu)$ is indeed diagonalizable. 
\end{proof}

Recall from the definition of admissible multiplier system that the exponent $\Lambda_{\cu}$ for any cusp $\cu$ of $\rG$ is a diagonal matrix such that $\mc{P}_{\cu}^{-1} \rho(t_{\cu}) \mc{P}_{\cu}=\exp(2\pi \Ii \Lambda_{\cu})$ for some diagonalizing matrix $\mc{P}_{\cu}$.
\begin{cor}\label{Omega}
For any cusp $\cu$ of $\rG$ an exponent $\Omega_{\cu}$ of the induced representation $\mr{Ind}_{_{\mr{H}}}^{^{\rG}}(\rho)$, of a rank $d$ admissible representation $\rho$ of $\mr{H}$,  has components $\frac{(\Lambda_{_{i}})_{kk}+j}{h_{i}}$, where $1\leq i \leq \n_{\cu}$, $0\leq j < h_{i}$, $1\leq k\leq d$ and $\Lambda_{i}$ is an exponent of $\rho$ at cusp $\cu_{i}$.
\end{cor}
\begin{proof} From the proof of Theorem~\ref{adm-ind}, for every $\cu \in \widehat{\Cu_{_{\rG}}}$, $\widetilde{\rho}(t_{\cu})$ is a diagonalizable matrix with the eigenvalues $$\bigg\{ \xi \lambda_{_{(i,k)}}^{^{1/h_{i}}} \ \big|\ 1\leq i \leq \n_{\cu}, 1\leq k \leq d, \mr{and}\ \xi=\exp\bigg(\frac{2\pi \Ii j}{h_{i}}\bigg), 1\leq j < h_{i} \bigg\}\,.$$  For all $i$ there exists a diagonalizing matrix  $\mc{P}_{i}$ such that $\mc{P}_{i}^{-1} \rho(t_{i}) \mc{P}_{i} = \exp(2\pi \Ii \Lambda_{i})$, where  the exponent matrix $\Lambda_{i}=\mr{Diag}(\Lambda_{i1}, \ldots, \Lambda_{id})$. Therefore,  $\lambda_{_{(i,k)}}= \exp(2\pi \Ii \Lambda_{_{ik}})$. This implies that $\xi \lambda_{_{(i,k)}}^{^{1/h_{i}}} = \exp\bigg(2 \pi \Ii\frac{\Lambda_{_{ik}}+j}{h_{i}}\bigg)$. Hence the exponent $\Omega_{\cu}$ of $\widetilde{\rho}(t_{\cu})$ has $dm$ diagonal entries of the form $\frac{(\Lambda_{_{i}})_{kk}+j}{h_{i}}$. 
\end{proof}

A formal proof of Theorem~\ref{adm-ind} in complete generality made the argument more complicated than it really is. Thus, this simple idea will be illustrated with an example in section~\ref{examples}. 

\begin{proof}[Proof of Theorem~\ref{induction}] $\widetilde\rho : \rG \rightarrow \mr{GL}_{dm}(\C)$ is an induced representation of $\rG$ of an admissible representation $\rho: \mr{H} \rightarrow \mr{GL}_{d}(\C)$ of $\mr{H}$. For any representation $\rho: \mr{H} \rightarrow \mr{GL}_d(\C)$ we wish to find  an isomorphism between $\m_{w}^{!}(\rho) $ and $\m_{w}^{!}(\widetilde\rho)$. Lemma~\ref{weight0 relation} gives $\m_{w}^{!}(\rho) \approx \m_{0}^{!}(\rho \otimes \nu_{_{\mr{H}}}^{^{-w}})$ using the isomorphism $\X(\tau) \mapsto \Delta_{_\rG}^{^{-\frac{w}{2\fL}}} \X(\tau)$, where $\nu_{_{\mr{H}}}^{^{-w}}$ is the restriction of $\nu_{_{\rG}}^{^{-w}}$ to $\mr{H}$. Similarly, $\m_{w}^{!}(\widetilde\rho) \approx \m_{0}^{!}(\widetilde\rho \otimes \nu_{_{\rG}}^{^{-w}})$. Therefore to show a one-to-one correspondence between $\m_{w}^{!}(\rho)$ and $\m_{w}^{!}(\widetilde\rho)$ it is enough to establish a one-to-one correspondence between $\m_{0}^{!}(\rho \otimes \nu_{_{\mr{H}}}^{^{-w}})$ and $\m_{0}^{!}(\widetilde\rho \otimes \nu_{_{\rG}}^{^{-w}})$. Note that $\mr{Ind}_{_{\mr{H}}}^{^{\rG}}(\rho \otimes \nu_{_{\mr{H}}}^{^{-w}}) = \mr{Ind}_{_{\mr{H}}}^{^{\rG}}(\rho) \otimes \nu_{_{\rG}}^{^{-w}}$.  
Let $\X(\tau) \in \m_{0}^{!}(\rho)$ then define $$\widetilde{\X}(\tau) =\bigg(\X(\g_{1}^{-1}\tau), \X(\g_{2}^{-1}\tau), \ldots, \X(\g_{m}^{-1}\tau)\bigg)^{\mathfrak{t}}.$$ We claim that $\widetilde{\X}(\tau) \in \m_{0}^{!}(\widetilde\rho^{\ '})$. Since every component is weakly holomorphic therefore $\widetilde{\X}(\tau)$ is also weakly holomorphic. Hence it suffices to check the functional behaviour of $\widetilde{\X}(\tau)$ under $\rG$\,, i.e. for all $\g=\pm\tmt{a}{b}{c}{d} \in \rG, \ \widetilde{\X}(\g \tau) = \widetilde\rho(\g)\widetilde{\X}(\tau).$ Consider $\widetilde{\X}(\g\tau)$ for $\g \in \rG$, then by definition 
\begin{eqnarray*} \footnotesize
\widetilde{\X}(\g\tau) = \left(\begin{array}{ccc}
\X(\g_{1}^{-1}\g\tau)\\
\X(\g_{2}^{-1}\g\tau)\\ 
\vdots\\ 
\X(\g_{m}^{-1}\g\tau)\\
\end{array} \right) 
 = \left( \begin{array}{ccc}
\X(\g_{1}^{-1}\g\g_{j_{1}}\g_{j_{1}}^{-1}\tau)\\
\X(\g_{2}^{-1}\g \g_{j_{2}}\g_{j_{2}}^{-1}\tau) \\
\vdots\\
\X(\g_{m}^{-1}\g \g_{j_{m}}\g_{j_{m}}^{-1}\tau)\\
\end{array} \right) =  \left( \begin{array}{ccc}
\rho(\g_{1}^{-1}\g \g_{j_{1}}) \X(\g_{j_{1}}^{-1}\tau) \\
\rho(\g_{2}^{-1}\g \g_{j_{2}}) \X(\g_{j_{2}}^{-1}\tau) \\
\vdots\\
\rho(\g_{m}^{-1}\g \g_{j_{m}}) \X(\g_{j_{m}}^{-1}\tau) \\
\end{array} \right) \,.\end{eqnarray*} This implies that $\widetilde{\X}(\tau) \in \m_{0}^{!}(\widetilde\rho)$. Conversely, for any $\widetilde{\X}(\tau) \in \m_{0}^{!}(\widetilde\rho)$ define $\X(\tau)$ by taking the first $d$ components of $\widetilde{\X}(\tau)$, \ie $\X(\tau)=\bigg(\widetilde{\X}_{1}(\tau), \ldots, \widetilde{\X}_{d}(\tau) \bigg)^{\mathfrak{t}}$. Since $\g_1=1$ therefore for every $\g \in \mr{H}$, $\rho(\g)$ will appear as the first $d\times d$ block in the $dm\times dm$ matrix $\tilde{\rho}(\g)$ such that all the other entries in the first row and column are zeros and first $d$ components on both sides of $\widetilde{X}(\g\tau)=\widetilde{\rho}(\g) \widetilde{X}(\tau),$ for every $\g \in \mr{H}$ give the required identity $\X(\g \tau)=\rho(\g) \X(\tau),$ for every $\g \in \mr{H}$. To see whether thus defined $\X(\tau)$ will have Fourier expansion at every cusp of H, first notice that $\widehat{\Cu}_{_\mr{H}}=\{ \g_{j}^{-1} \cu\ |\ \cu \in \widehat{\Cu}_{_\rG}\ \&\ 1\leq j \leq m \}$ and $\widehat{\Cu}_{_\rG} \subset \widehat{\Cu}_{_\mr{H}}$. Since for all $j, \g_{j} \notin \mr{H}$ we obtain $$\X(\g_{j}\tau) = \bigg(\widetilde{\X}_{1}(\g_{j}\tau), \widetilde{\X}_{2}(\g_{j}\tau), \ldots, \widetilde{\X}_{d}(\g_{j}\tau) \bigg)^{\mathfrak{t}}.$$
Since for any $\cu \in \widehat{\Cu_{_\rG}}$, $\cu$ and $\g_{j}^{-1}\cu$  are $\rG$-equivalent cusps, every component of $\X(\g_{j} \tau)$ inherits the Fourier expansion at cusp $\cu$ from the Fourier expansion of $\widetilde{\X}(\g_{j} \tau)$. Hence, $\widetilde{\X}(\tau)$ is a weakly holomorphic vvmf and has Fourier expansion at every cusp of $\rG$.  
\end{proof}

\section{Existence}
\begin{thm}\label{thm:existence-for-fuchsian}
Let $\rG$ be a Fuchsian group of the first kind and $\rho:\rG \rightarrow \mr{GL}_d(\C)$ be any admissible representation of finite image. Then there exists a weakly holomorphic vector-valued modular function for $\rG$ with multiplier $\rho$, whose components are linearly independent over $\C$. 
\end{thm}
\begin{proof}
First, note that if $f(z)$ is any nonconstant function holomorphic in some disc then the powers $f(z)^1,f(z)^2,...$ are linearly independent over $\C$. To see this let $z_0$ be in the disc; it suffices to prove this for the powers of $g(z)=f(z)-f(z_0)$, but this is clear from Taylor series expansion of $g(z)=\sum_{n=k}^\infty (z-z_0)^n a_n$ where $a_k \ne 0$ ($k$ is the order of the zero at $z=z_0$). In particular, if $f(z)$ is any nonconstant modular function for any Fuchsian group of the first kind then its powers are linearly independent over $\C$.

Moreover, suppose  $\rG,\mr{H}$ are distinct Fuchsian groups of the first kind with $\mr{H}$ normal in G with index m. Fix any $\tau_0\in \h\backslash\{\mc{E}_{\mr{G}}\}$ such that all $m$ points $\gamma_i\tau_0$ are distinct where $\gamma_i$ are $m$ inequivalent coset representatives. Then there is a modular function $f(\tau)$ for $\mr{H}$ such that  the $m$ points $f(\gamma_i\tau_0)$ are distinct. This is because distinct Fuchsian groups must have distinct sets of modular functions. Define \begin{equation}\label{eq:gtau} g(\tau)=\prod_i(f(\gamma_i\tau)-f(\gamma_i\tau_0))^i . \end{equation} Then $g(\tau)$ is also a modular function for $\mr{H}$, and manifestly the $m$ functions $g(\gamma_i\tau)$ are linearly independent over $\C$ (since they have different orders of vanishing at $\tau_0$).

Let $\mr{H}=\ker(\rho)$. Then $\rho$ defines a representation of the finite group $\mr{K}=\rG/\mr{H}$, so $\rho$ decomposes into a direct sum $\oplus_im_i\rho_i$ of irreducible representations $\rho_i$ of K, where $m_i$ is the multiplicity of  irreducible representation $\rho_{i}$ of K in $\rho$. 

Suppose that the Theorem is true for all  irreducible representations $\rho_i$ of K. Let $$\X_i(\tau)=\bigg(\X_{i1}(\tau),\cdots,
\X_{id_i}(\tau)\bigg)^{\mathfrak{t}}$$ be a vvmf for the $i^{th}$-irreducible representation of K with linearly independent components. Changing basis, $\rho$ can be written in the block-diagonal form ($m_i$ blocks for each $\rho_i$). Choose any nonconstant modular function $f(\tau)$ of G. Then $$\X(\tau)=\bigg(f(\tau)\X_1(\tau),\cdots,f(\tau)^{m_1}\X_1(\tau),f(\tau)\X_2(\tau),\cdots,f(\tau)^{m_2}\X_{2}(\tau),\cdots \cdots\bigg)^{\mathfrak{t}}$$ will be a vvmf for G with multiplier $\rho$ (or rather $\rho$ written in block-diagonal form), and the components of $\X(\tau)$ will be linearly independent over $\C$. 

So it suffices to prove the theorem for  irreducible representations of K. Let $m=[\rG:\mr{H}]=|\mr{K}|$ and write $\rG=\g_1 \mr{H}\cup \g_2 \mr{H} \cup \cdots \cup \g_m \mr{H}.$ Let $g(\tau)$ be the modular function for $\mr{H}$ defined above by equation~\eqr{gtau}, which is such that the $m$ functions $g(\g_i\tau)$ are linearly independent over $\C$. Induce $g(\tau)$ (which transforms by the trivial $\mr{H}$-representation)
from $\mr{H}$ to a vvmf $\X_{_{g}}(\tau)$ of G; by definition its $m$ components $\X_{_{g,i}}(\tau)=g(\g_i \tau)$ are linearly independent over $\C$. Inducing the trivial representation of $\mr{H}$ gives the regular representation of K and the regular representation of a finite group (such as K) contains each  irreducible representation (in fact with a multiplicity equal to the dimension of the  irreducible representation). To find a vvmf for the K-irreducible representation $\rho_j$ find a subrepresentation of regular representation equivalent to $\sigma_j$ for some $j$ in $\sigma:=\mr{Ind}_{_{\mr{H}}}^{^{\rG}}(1) = \oplus_{j} m_{j} \sigma_{j}.$ Observe that every $\X(\tau)$  of $\m^{{!}}_{0}(\sigma)$ can be realized as $\X_{g}(\tau)$ for some $g(\tau)$ of $\m^{{!}}_{0}(1)$ where $1$ denotes the trivial representation of $\mr{H}$. Let $P$ be a change of basis matrix such that $P^{-1} \sigma P = \delta \oplus \delta'$ with $\delta = \sigma_{j}$, and $\sigma_{j}$  appearing as first $d_{j } \times d_{j}$ block matrix in the matrix $P^{-1} \sigma P $. Here $P$ can be realized as a block permutation matrix of the summands $\sigma_{j}$'s of $\sigma$. Now,  construct $\X_{j}(\tau)$ by the first $d_j$ components of $P\X(\tau)$ which is the desired vvmf of $\rho_j$. By construction all the components of $\X_{j}(\tau)$ will be linearly independent over $\C$.
\end{proof}

\section{Examples}\label{examples}
Let us fix $t=\pm \tmt{1}{1}{0}{1}, s=\pm \tmt{0}{1}{-1}{0} $ and $u= st^{-1}=\pm \tmt{0}{-1}{1}{-1}.$  The matrices $t,s$ and $u$ are of order $\infty$, $2$ and $3$ respectively.  Write
 \begin{eqnarray*}
\rG \hspace{-0.4cm}&:=\G(1) \hspace{-0.2cm} &\cong \big\langle t, s, u \ \big| \   s^2 =1=u^3 =tsu \big\rangle, \\
\mr{H} \hspace{-0.4cm}&\,\,\, := \G_{0}(2) \hspace{-0.2cm}&\cong  \big\langle t_{\infty}:=t, t_0 :=s t^2 s, t_{\omega}\ \big| \ t_{\omega}^{2}=1=t_{\omega} t_{0} t_{\infty} \big\rangle, \\
\mr{K} \hspace{-0.4cm}& :=\G(2) \hspace{-0.2cm}&\cong  \big\langle t_\infty:=t^{2}, t_0:=s t^2 s, t_{1}:=t s t^2 s t^{-1} \ \big| \  t_{1} t_{\infty}t_{0} =1\big\rangle, 
\end{eqnarray*}
where  $t_{\omega}=\pm\tmt{\ 1}{\ 1}{-2}{\ -1}$ and $\om=\frac{-1+i}{2}$ is an elliptic fixed point of order $2$.  $\mr{H}$ and $\mr{K}$ both are congruence subgroups of $\rG$ of index 3 and 6 respectively and recall that  $\widehat{\Cu}_{_\rG}=\{ \infty \}, \ \widehat{\Cu}_{_\mr{H}}=\{ \infty, 0 \}, \ \widehat{\Cu}_{_\mr{K}}=\{ \infty, 0, 1 \}.$ 
\subsection{Exponent matrix of a lift}\label{exponent matrix of a lift}
\begin{itemize} 
\item Since $[\rG:\mr{K}]=6$ write $\rG= \mr{K} \cup t\mr{K} \cup s \mr{K} \cup ts \mr{K} \cup st\mr{K} \cup tst\mr{K}$.  Let $\rho:\mr{K} \rightarrow \mr{GL}_{d}(\C)$ be any admissible representation  and write $\rho(t_{\cu})=T_{\cu}$ for the cusp $\cu=1,0,\infty$, then there exist exponent matrices $\Lambda, \Lambda_0, \Lambda_{1}$ and matrices $\mc{P}_0, \mc{P}_{1} \in \mr{GL}_{d}(\C)$ such that $T_\infty=\exp(2\pi \Ii \Lambda)$, $\mc{P}_{0} T_{0} \mc{P}_{0}^{-1}=\exp(2\pi \Ii \Lambda_{0})$, $\mc{P}_{1} T_{1} \mc{P}_{1}^{-1}=\exp(2\pi \Ii \Lambda_{_{1}})$ are diagonal matrices, where $\Lambda,\Lambda_{0},$ and $\Lambda_{_{1}}$ are  exponent matrices of cusps $\infty, 0$ and $1$ respectively. Define $\widehat{\rho}=\mr{Ind}_{_{\mr{K}}}^{^{\rG}} (\rho):\mr{G} \rightarrow \mr{GL}_{6d}(\C) $, then  from equation~\eqr{indrep}  \begin{equation*}  \widehat{\rho}(t) :=\widehat{T}_{\infty} = \left( \begin{array}{cccccccc}
0 &T_{\infty}& 0&0 &0& 0\\
I&0&0&0&0&0 \\
0&0&0&T_{0}&0&0\\
0&0&I&0&0&0\\
0&0&0&0&0&T_{_{1}}\\
0&0&0&0&I&0
\end{array} \right)\,.  \end{equation*} To assure the admissibility of $\widehat{\rho}$ we need to show that  $\widehat{\rho}(t)$ is diagonalizable, and   this follows from Theorem~\ref{adm-ind}. From Corollary~\ref{Omega}, the exponent matrix of cusp $\infty$ of $\rG$ with respect to $\widehat{\rho}$ is $$\Omega= \mr{Diag}\bigg( \frac{\Lambda}{2}, \frac{1+\Lambda}{2}, \frac{\Lambda_{0}}{2},\frac{1+\Lambda_{0}}{2}, \frac{\Lambda_{1}}{2}, \frac{1+\Lambda_{_{1}}}{2}\bigg).$$  

\item Since $[\rG:\mr{H}]=3$ write $\rG= \mr{H} \cup s\mr{H} \cup ts \mr{H} .$  Let $\rho:\mr{H} \rightarrow \mr{GL}_{d}(\C)$ be any admissible representation. From the definition of induced representation  $\widetilde{\rho}=\mr{Ind}_{_{\mr{H}}}^{^{\rG}} (\rho) :\rG \rightarrow \mr{GL}_{3d}(\C)$ defined by the equation~\eqr{indrep} \begin{equation*} \widetilde{\rho}(t) :=\widetilde{T}_{\infty} = \left( \begin{array}{cccccccc}
T_{\infty}& 0&0\\
0&0&T_{0} \\
0&I&0
\end{array} \right)\, .\end{equation*} From Theorem~\ref{adm-ind}, it follows that $\widetilde{\rho}(t)$ is diagonalizable and from Corollary~\ref{Omega}, the exponent matrix of cusp $\infty$ of $\rG$ with respect to the admissible representation $\widehat{\rho}$ is $$\Omega= \mr{Diag}\bigg( \Lambda, \frac{\Lambda_{0}}{2}, \frac{1+\Lambda_{0}}{2} \bigg)\, $$

\end{itemize}

\subsection{An easy construction of vvmf}\label{G2andG02}
An explaination of the above ideas is provided by constructing a rank two and three vvmf of $\G_{0}(2)$ and $\G(1)$. Write $\mr{H}=\mr{K} \cup t\mr{K}$.  Let $\sigma: \mr{K}\longrightarrow \C^{\times}$ be a trivial multiplier of $\mr{K}$. Consider $\widetilde\rho=\mr{Ind}_{_{\mr{K}}}^{^{\mr{H}}}(\sigma): \mr{H}\rightarrow \mr{GL}_{2}(\C)$ to be the induced  representation of $\mr{H}$. Therefore, by equation~\eqr{indrep} we get $\widetilde{T}_{\infty}=\widetilde{\rho}(t)= \left({0 \atop 1}{ 1\atop 0 } \right)$ and $\widetilde{T}_{0} =\widetilde{\rho}(t_{0})=\left( {1 \atop 0} { 0 \atop 1 }  \right).$ Let $\z_{_\mr{K}}(\tau)=-\frac{1}{16}\big( \tq^{^{\ -1}} -8+20 \tq -62\tq^{^{3}}+216 \tq^{^{5}}+ \cdots \big)$ with $q=\exp(2\pi \Ii \tau)$ and $\tq=q^{1/2}$, be a hauptmodul of $\mr{K}$ which sends the cusps $\infty$, $0$ and $1$  respectively to $\infty$, $0$ and $1$. Consider $\X(\tau)=\z_{_\mr{K}}(\tau)$, a weight zero scalar-valued modular form of $\mr{K}$, then $$\widetilde{\X}(\tau)=\bigg( \X(\tau), \X(t^{-1}\tau) \bigg)^{\mathfrak{t}}$$ is a weight zero rank two vvmf of $\rG$ with respect to an equivalent admissible multiplier $\widetilde{\rho}\ '=\mc{P}^{-1}\widetilde{\rho}\ \mc{P}$ of $\widetilde\rho$, where $\mc{P} =\tmt{1}{\ -1}{1}{\ \ 1}$ and the exponent matrix $\Omega$  of $\widetilde{\rho}\ '$ is $\tmt{1/2}{0}{0}{1}$. 
\par Similarly, write $\rG=\mr{H} \cup s\mr{H} \cup (t s)\mr{H}$. Let $\z_{_{\mr{H}}}(\tau)=-\frac{1}{64}\big( q^{-1}-24+276q-2048q^{2}+\cdots\big)$ be a hauptmodul of H which takes the values $\infty$, $0$ and $1$ respectively at $\infty$, $0$ and $\omega$. Let $\mathfrak{1}:\mr{H}\rightarrow \C^{\times}$ be the trivial multiplier of H and $\widehat{\rho}=\mr{Ind}_{_{\mr{H}}}^{^{\rG}}(\mathfrak{1}):\rG\rightarrow \mr{GL}_3(\C)$ be the induced representation of $\rG$, then for $\X(\tau)=\z_{_{\mr{H}}}(\tau)$,
 $$\widetilde{\X}(\tau)=\bigg(\X(\tau), \X(s^{-1}\tau), \X((t s)^{-1}\tau) \bigg)^{\mathfrak{t}}$$ is a weight zero rank three vvmf of $\G(1)$ with respect to an equivalent admissible multiplier $\widehat{\rho}\ '$ of $\widehat{\rho}$  where following equation~\eqr{indrep}, $$ \widehat{T}_{\infty}=\widehat{\rho}(t)=\left( \begin{array}{cccc}
\mathfrak{1}(t) &0&0\\
0&0&\mathfrak{1}(t_{0})\\
0&\mathfrak{1}(1)&0\\
\end{array} \right)=\left( \begin{array}{cccc}
1 &0&0\\
0&0&1\\
0&1&0\\
\end{array} \right)\,$$ 
and therefore $\widehat{\rho}\ '= \mc{P}^{-1} \widehat{\rho}\ \mc{P}$ with
$\mc{P}=\left( \begin{array}{cccc}
1 &0&0\\
0&1&-1\\
0&1&1
\end{array} \right).$
In this case, an exponent matrix $\Omega$ of $\widehat{\rho}\ '$ is the diagonal matrix $\mr{Diag}(1,1,1/2).$
\section*{Acknowledgements}
The author wants to thank the Max Planck Institute for Mathematics in Bonn and 
Mathematics Institute of Georg-August University G\"ottingen  for providing a rich and stimulating environment to pursue research. This article is based on a chapter of the author's Ph.D thesis written at the University of Alberta, Canada. The author wishes to thank Terry Gannon for his guidance throughout and many valuable comments during the preparation of this article. 
\nocite{}

 \bibliographystyle{abbrv}
 \bibliography{induction}
\end{document}